\documentclass[oneside,english]{amsart}
\usepackage[T1]{fontenc}
\usepackage[latin9]{inputenc}
\usepackage{amsthm}
\usepackage{amstext}
\usepackage[all]{xy}

\makeatletter
\numberwithin{equation}{section}
\numberwithin{figure}{section}
\theoremstyle{plain}
\newtheorem{thm}{\protect\theoremname}[section]
  \theoremstyle{plain}
  \newtheorem{lem}[thm]{\protect\lemmaname}

\makeatother

\usepackage{babel}
  \providecommand{\lemmaname}{Lemma}
\providecommand{\theoremname}{Theorem}

\begin{document}

\title{A folk Quillen model structure for operads}

\author{Ittay Weiss}
\begin{abstract}
We establish, by elementary means, the existence of a cofibrantly
generated monoidal model structure on the category of operads. By
slicing over a suitable operad the classical Rezk model structure
on the category of small categories is recovered. 
\end{abstract}

\maketitle

\section{Introduction}

The aim of the article is to present an elementary construction of
a Quillen model structure on the category ${\bf Ope_{\Sigma}}$ of
symmetric operads and on the category ${\bf Ope}$ of non-symmetric
operads where, in each case, every operad is fibrant (and cofibrant).
The existence of this model structure was announced without proof
at the end of Section 8 of \cite{MR2366165}. The same model structure
can be deduced from two more recent, and much more sophisticated,
results. One is the Cisinski-Moerdijk model structure on simplicial
operads \cite{MR3100888} and the other result is Giovanni Caviglia's
\cite{GiovCav}, establishing a model structure on enriched operads.
The non-enriched model structure on operads is relatively simple to
establish and does not require complicated machinery like Quillen's
small object argument. We thus give an explicit and elementary construction
of the model structures, together with explicit generating cofibrations.
In that sense the model structures belong to what are commonly referred
to as folk model structures. We also note that by slicing either operads
category over a suitable operad, one obtains the category ${\bf Cat}$
of small categories. When the model structure on operads is transferred
to the sliced category one recovers Rezk's folk model structure on
${\bf Cat}$. We also show that the model structure on ${\bf Ope_{\Sigma}}$
is a monoidal model structure when considering the Boardman-Vogt tensor
product. 

After briefly recalling the main concepts of Quillen model structures,
operads, and Rezk's model structure on ${\bf Cat}$, the second section
contains the proofs of the main results of this article.

\subsection{Quillen model structures}

A \emph{Quillen model structure }on a category $\mathcal{C}$ is a
specification of three classes of morphisms, called \emph{weak equivalences},
\emph{fibrations}, and \emph{cofibrations,} such that the following
axioms hold. 
\begin{itemize}
\item $\mathcal{C}$ is small complete and small cocomplete.
\item The weak equivalences satisfy the \emph{three for two} property, namely
if $h=g\circ f$ and any two of the three morphisms is a weak equivalence,
then so is the third. 
\item Each of the specified classes of morphisms is closed under retracts.
\item Liftings exist. In more detail, consider a commutative square 
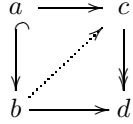
\[
\xymatrix{*++{a}\ar[r]\ar@{^{(}->}[d] & *++{c}\ar@{->>}[d]\\
b\ar[r]\ar@{..>}[ur] & d
}
\]
where the left vertical arrow is a cofibration and the right vertical
arrow is a fibration. If any of these morphisms is a weak equivalence,
then a diagonal dotted arrow exists, making the diagram commutative. 
\item Factorizations exist. In more detail, every morphism $f:a\to b$ can
be factorized as $f=g\circ h$ where $g$ is a trivial fibration (i.e.,
a fibration that is also a weak equivalence) and $h$ is a cofibration,
and $f$ can also be factorized as $f=g'\circ h'$ where $g'$ is
a fibration and $h'$ is a trivial cofibration (i.e., a cofibration
that is also a weak equivalence). 
\end{itemize}
We note that typically the verification of the first three axioms
is straightforward. It is the reconciliation between the liftings
axiom and the factorizations axiom that requires a fine-tuned balance
between the three classes of morphisms.

If the category $\mathcal{C}$ is equipped with a monoidal structure,
then a Quillen model structure on $\mathcal{C}$ with every object
cofibrant is compatible with the tensor product if the following pushout-product
axiom is satisfied. For any two morphisms $F:a\to b$ and $G:a'\to b'$,
consider the diagram
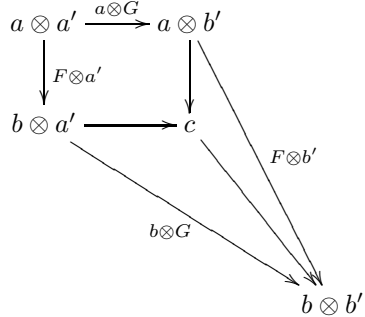
\[
\xymatrix{a\otimes a'\ar[r]^{a\otimes G}\ar[d]^{F\otimes a'} & a\otimes b'\ar[dddr]^{F\otimes b'}\ar[d]\\
b\otimes a'\ar[rrdd]_{b\otimes G}\ar[r] & c\ar[rdd]\\
\\
 &  & b\otimes b'
}
\]
where the square is a pushout, and the corner map $F\wedge G:c\to b\otimes b'$
is the induced one from the pushout. The pushout-product axiom states
that if both $F$ and $G$ are cofibrations, than so is $F\wedge G$,
and, moreover, if any of the given cofibrations is a trivial cofibration,
then $F\wedge G$ is a trivial cofibration. For more details on Quillen
model categories see the classical \cite{MR1361887}, the comprehensive
\cite{MR1944041}, or the recent survey \cite{MR3101400}.

\subsection{Operads}

A \emph{non-symmetric operad }$\mathcal{P}$, also known as a \emph{multicategory},
is a collection ${\rm ob}(\mathcal{P})$ of \emph{objects }and, for
all objects $p_{0},\ldots,p_{n}$, $n\ge0$, a set\emph{ }$\mathcal{P}(p_{1},\ldots,p_{n};p_{0})$,
also called a \emph{hom-set}. The elements in each hom set are referred
to as \emph{morphisms} and are also denoted by $\psi:p_{1},\ldots,p_{n}\to p_{0}$
instead of $\psi\in\mathcal{P}(p_{1},\ldots,p_{n};p_{0}$). There
is further a specified rule for composing morphisms when their domains
and codomains suitably match. In more detail, if $\psi:p_{1},\ldots,p_{n}\to p_{0}$
is a morphism and, for each $1\le i\le n$, $\psi_{i}:q_{1}^{i},\ldots,q_{k_{i}}^{i}\to p_{i}$
is a morphism, then a morphism $\psi\circ(\psi_{1},\ldots,\psi_{n}):q_{1}^{1},\ldots,q_{k_{1}}^{1},\ldots,q_{1}^{n},\ldots,q_{k_{n}}^{n}\to p_{0}$
is designated which is defined to be the composition of the given
morphisms. The composition is required to be associative in the obvious
sense and there are also \emph{identity }morphisms, that is for each
object $p$ there is an identity morphism ${\rm id}_{p}:p\to p$,
which behaves like an identity with respect to the composition. A
\emph{symmetric operad }is an operad where for all $p_{0},\ldots,p_{n}$
and a permutation $\sigma\in\Sigma_{n}$, there is a function $\sigma^{*}:\mathcal{P}(p_{1},\ldots,p_{n};p_{0})\to\mathcal{P}(p_{\sigma(1)},\ldots,p_{\sigma(n)};p_{0})$,
and these functions are required to be compatible with the composition
operation in the obvious manner. For both symmetric and non-symmetric
operads, the notion of a structure preserving operation between operads
is referred to as a \emph{functor}. A functor $F:\mathcal{P}\to\mathcal{Q}$
is thus a function $F:{\rm ob}(\mathcal{P})\to{\rm ob}(\mathcal{Q})$
together with, for all objects $p_{0},\ldots,p_{n}$, a function $F:\mathcal{P}(p_{1},\ldots,p_{n};p_{0})\to\mathcal{Q}(Fp_{1},\ldots,Fp_{n};Fp_{0})$,
where, of course, the composition and identities are to be respected. 

We note that any category $\mathcal{C}$ gives rise to an operad $j_{!}(\mathcal{C})$
where ${\rm ob}(j_{!}(\mathcal{C}))={\rm ob}(\mathcal{C})$. The only
morphisms in $j_{!}(\mathcal{C})$ are unary morphisms, given by $j_{!}(\mathcal{C})(a;b)=\mathcal{C}(a,b)$,
with identities and composition as in $\mathcal{C}$. Obviously, $j_{!}(\mathcal{C})$
is also a symmetric operad in a unique way. We thus obtain two functors
$j_{!}:{\bf Cat}\to{\bf Ope}$ and $j_{!}:{\bf Cat}\to{\bf Ope_{\Sigma}}$.
In the other direction, any operad has an underlying category obtained
by neglecting the non-unary morphisms, giving rise to the functors
$j^{*}:{\bf Ope}\to{\bf Cat}$ and $j^{*}:{\bf Ope_{\Sigma}}\to{\bf Cat}$.
It is easy to see that $j_{!}$ is left adjoint to $j^{*}$. By means
of $j^{*}$, the terminology of category theory applies to operads,
in particular, a morphism in an operad is an isomorphism if the morphism
survives in $j^{*}(\mathcal{P})$ and is an isomorphism there. 

Let $\mathcal{P}$ and $\mathcal{Q}$ be two symmetric operads. The
\emph{Boardman-Vogt tensor product} of these operads is the symmetric
operad $\mathcal{P}\otimes_{BV}\mathcal{Q}$ with ${\rm ob}(\mathcal{P}\otimes_{BV}\mathcal{Q})={\rm ob}(\mathcal{P})\times{\rm ob}(\mathcal{Q})$
given in terms of generators and relations as follows. Let $C$ be
the collection on ${\rm ob}(\mathcal{P})\times{\rm ob}(\mathcal{Q})$
which contains the following generators. For each $q\in{\rm ob}(\mathcal{Q})$
and each morphism $\psi\in\mathcal{P}(p_{1},\cdots,p_{n};p)$, there
is a generator $\psi\otimes_{bv}q$ in $C((p_{1},q),\cdots,(p_{n},q);(p,q))$.
Similarly, for each $p\in{\rm ob}(\mathcal{P})$ and a morphism $\varphi\in\mathcal{Q}(q_{1},\cdots,q_{m};q)$,
there is a generator $p\otimes_{bv}\varphi$ in $C((p,q_{1}),\cdots,(p,q_{m});(p,q))$.
There are five types of relations among the arrows generated by these
generators:
\begin{enumerate}
\item $(\psi\otimes_{bv}q)\circ((\psi_{1}\otimes_{bv}q),\cdots,(\psi_{n}\otimes_{bv}q))=(\psi\circ(\psi_{1},\cdots,\psi_{n}))\otimes_{bv}q$
\item  $\sigma^{*}(\psi\otimes_{bv}q)=(\sigma^{*}\psi)\otimes_{bv}q$
\item $(p\otimes_{bv}\varphi)\circ((p\otimes_{bv}\varphi_{1}),\cdots,(p\otimes_{bv}\varphi_{m}))=p\otimes_{bv}(\varphi\circ(\varphi_{1},\cdots,\varphi_{m}))$
\item $\sigma^{*}(p\otimes_{bv}\varphi)=p\otimes_{bv}(\sigma^{*}\varphi)$
\item $(\psi\otimes_{bv}q)\circ((p_{1}\otimes_{bv}\varphi),\cdots,(p_{n}\otimes_{bv}\varphi))=\sigma_{m,n}^{*}((p\otimes_{bv}\varphi)\circ((\psi,q_{1}),\cdots,(\psi,q_{m})))$
\end{enumerate}
By the relations above we mean every possible choice of morphisms
for which the compositions are defined. The relations of type 1 and
2 ensure that for any $q\in{\rm ob}(\mathcal{P})$, the map $p\mapsto(p,q)$
naturally extends to a map of operads $\mathcal{P}\rightarrow\mathcal{P}\otimes_{BV}\mathcal{Q}$.
Similarly, the relations of type 3 and 4 guarantee that for each $p\in{\rm ob}(\mathcal{P})$,
the map $q\mapsto(p,q)$ naturally extends to a map of operads $\mathcal{Q}\rightarrow\mathcal{P}\otimes_{BV}\mathcal{Q}$.
The relation of type 5 can be pictured as follows. The left hand side
is a morphism in the free operad, represented by the labelled planar
tree
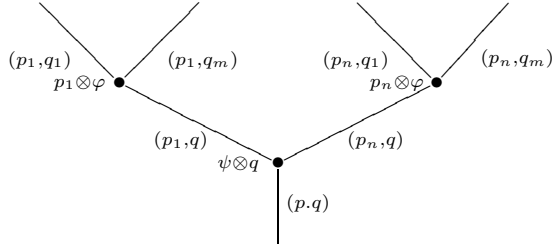
\[
\xymatrix{*{}\ar@{-}[dr]_{(p_{1},q_{1})} &  & *{}\ar@{-}[dl]^{(p_{1},q_{m})} &  & *{}\ar@{-}[dr]_{(p_{n},q_{1})} &  & *{}\ar@{-}[dl]^{(p_{n},q_{m})}\\
\ar@{}[r]|{p_{1}\otimes\varphi} & *{\bullet}\ar@{-}[drr]_{(p_{1},q)} &  &  & \ar@{}[r]|{p_{n}\otimes\varphi} & *{\bullet}\ar@{-}[dll]^{(p_{n},q)}\\
 &  & \ar@{}[r]|{\psi\otimes q} & *{\bullet}\ar@{-}[d]^{(p.q)}\\
 &  &  & *{}
}
\]
while the right hand side is given by the tree 
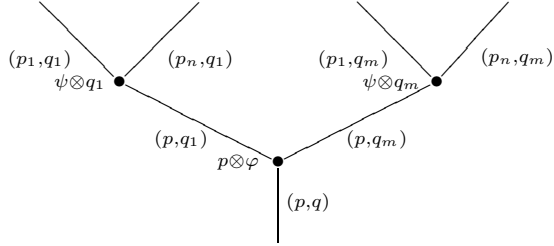
\[
\xymatrix{*{}\ar@{-}[dr]_{(p_{1},q_{1})} &  & *{}\ar@{-}[dl]^{(p_{n},q_{1})} &  & *{}\ar@{-}[dr]_{(p_{1},q_{m})} &  & *{}\ar@{-}[dl]^{(p_{n},q_{m})}\\
\ar@{}[r]|{\psi\otimes q_{1}} & *{\bullet}\ar@{-}[drr]_{(p,q_{1})} &  &  & \ar@{}[r]|{\psi\otimes q_{m}\,\,\,} & *{\bullet}\ar@{-}[dll]^{(p,q_{m})}\\
 &  & \ar@{}[r]|{p\otimes\varphi} & *{\bullet}\ar@{-}[d]^{(p,q)}\\
 &  &  & *{}
}
\]
before applying $\sigma_{m,n}^{*}$, which is the obvious permutation
that equates the domains of the two morphisms. With the Boardman-Vogt
tensor product, the category ${\bf Ope_{\Sigma}}$ is a closed monoidal
category. For an introduction to operads see \cite{MR2094071} or,
for a source that uses similar notation to the one above (and below),
and, in particular, proves the above made claims, see \cite{MR2742425}.

\subsection{The folk model structure on ${\bf Cat}$}

We briefly recount the definition of the Rezk model structure on ${\bf Cat}$. 
\begin{thm}
\label{thm:folkModelStru}The category ${\bf Cat}$ of small categories
admits a cofibrantly generated cartesian closed Quillen model structure
where:
\begin{itemize}
\item The weak equivalences are the categorical equivalences.
\item The cofibrations are those functors $F:\mathcal{C}\rightarrow\mathcal{D}$
that are injective on objects.
\item The fibrations are those functors $F:\mathcal{C}\rightarrow D$ such
that for any $C\in{\rm ob}(\mathcal{C})$ and each isomorphism $\psi:FC\rightarrow D$
in $\mathcal{D}$, there exists an isomorphism $\phi:C\rightarrow C'$
for which $F\phi=\psi$.
\end{itemize}
\end{thm}
The model structure is cofibrantly generated and compatible with the
cartesian product of categories. 

Recently, in \cite{BMhomThCat}, conditions are given for the existence
of a canonical model structure on the category of small categories
enriched in a monoidal category $\mathcal{V}$. The Rezk model structure
is recovered by the case $\mathcal{V}={\bf Set}$.

\section{The model structure on operads}

We now present the main result: an elementary presentation of a Quillen
model structure on operads. 
\begin{thm}
\label{thm:folkModelStrOnOperad}The categories ${\bf Ope}$ and ${\bf Ope_{\Sigma}}$
admit a Quillen model structure where:
\begin{itemize}
\item The weak equivalences are the operadic equivalences, namely those
functors $F:\mathcal{P}\to\mathcal{Q}$ which are essentially surjective
(i.e., such that $j^{*}(F)$ is essentially surjective) and fully
faithful (i.e., where each component function $F:\mathcal{P}(p_{1},\ldots,p_{n};p_{0})\to\mathcal{Q}(Fp_{1},\ldots,Fp_{n};Fp_{0})$
is a bijection). 
\item The cofibrations are those functors $F:\mathcal{P}\rightarrow\mathcal{Q}$
that are injective on objects.
\item The fibrations are those functors $F:\mathcal{P}\rightarrow\mathcal{Q}$
such that for any $p\in{\rm ob}(\mathcal{P})$ and each isomorphism
$\psi:Fp\rightarrow q$ in $\mathcal{Q}$, there exists an isomorphism
$\phi:p\rightarrow p'$ for which $F\phi=\psi$. 
\end{itemize}
\end{thm}
\begin{proof}
We treat both symmetric and non-symmetric operads together since the
symmetric group actions entail no complications. Thus we provide the
details for the non-symmetric operads noting at this point that whenever
one needs to extend the construction to include symmetric group actions,
the extension is the evident one. Notice that a functor $F:\mathcal{P}\rightarrow\mathcal{Q}$
is a fibration (respectively cofibration) if, and only if, $j^{*}F$
is a fibration (respectively cofibration) in the Rezk model structure.
Notice as well that a functor $F:\mathcal{P}\rightarrow\mathcal{Q}$
is a trivial fibration if, and only if, the function ${\rm ob}(F):{\rm ob}(\mathcal{P})\rightarrow{\rm ob}(\mathcal{Q})$
is surjective and $F$ is fully faithful. We now set out to prove
the Quillen axioms. Small limits, in both ${\bf Ope}$ and ${\bf Ope_{\Sigma}}$
are directly constructed much as they are constructed in the category
${\bf Cat}$, posing no difficulties. Showing the existence of small
colimits, namely of all small coproducts and all coequalizers, requires
some more care. Constructing small coproducts is trivial, but, viewing
operads as an extension of categories by means of $j_{!}$, operads
inherit the subtleties of categories. The details of coequalizers
of categories can be found in \cite{MR1725510} and, mutatis-mutandis,
the same construction gives rise to coequalizers of operads (both
symmetric and non-symmetric). The verification of the three for two
property and of closure under retracts is routine, and we thus turn
now to a detailed proof of the liftings and the factorizations axioms. 

Consider a commutative square
\[
\xymatrix{*++{\mathcal{P}}\ar[r]^{U}\ar@{^{(}->}[d]_{F} & *++{\mathcal{R}}\ar@{->>}[d]^{G}\\
\mathcal{Q}\ar[r]_{V}\ar@{..>}[ur]^{H} & \mathcal{S}
}
\]
where $F$ is a cofibration and $G$ is a fibration. We need to prove
the existence of a lift $H$ making the diagram commute, whenever
$F$ or $G$ is a weak equivalence. Assume first that $G$ is a weak
equivalence. Applying the object functor (that sends an operad $\mathcal{P}$
to the set ${\rm ob}(\mathcal{P})$) to the lifting diagram we obtain
\[
\xymatrix{*++{{\rm ob}(\mathcal{P})}\ar[r]^{U}\ar@{->}[d]_{F} & *++{{\rm ob}(\mathcal{R})}\ar@{->}[d]^{G}\\
{\rm ob}(\mathcal{Q})\ar[r]_{V}\ar@{..>}[ur]^{H} & {\rm ob}(\mathcal{S})
}
\]
where $F$ is injective and $G$ is surjective. We can thus find a
lift $H$. Let now $\psi\in\mathcal{Q}(q_{1},\cdots,q_{n};q)$, and
consider $V(\psi)\in\mathcal{S}(Vq_{1},\cdots,Vq_{n};Vq)$. Since
$G$ is fully faithful and $HG=V$ on the level of objects, we obtain
that the function 
\[
G:\mathcal{R}(Hq_{1},\cdots,Hq_{n};Hq)\rightarrow\mathcal{S}(Vq_{1},\cdots,Vq_{n};Vq)
\]
is an isomorphism. We now define $H(\psi)=G^{-1}(V(\psi))$. It is
easily checked that this (uniquely) extends $H$ and makes it into
the desired lift.

Assume now that $F$ is a trivial cofibration. We can thus construct
a functor $F':\mathcal{Q}\rightarrow\mathcal{P}$ such that 
\[
F'\circ F={\rm id}_{\mathcal{P}}
\]
 together with a natural isomorphism $\alpha:F\circ F'\rightarrow{\rm id}_{\mathcal{Q}}$
(the theory of natural transformations from a single functor to another
functor between operads is almost identical to that of natural transformations
between categories). We can moreover choose $\alpha$ such that for
each $p\in{\rm ob}(\mathcal{P})$, the component at $Fp$ is given
by
\[
\alpha_{Fp}={\rm id}_{Fp}.
\]
To define $H:{\rm ob}(\mathcal{Q})\rightarrow{\rm ob}(\mathcal{R})$,
let $q\in{\rm ob}(\mathcal{Q})$ and consider the object $VFF'q\in{\rm ob}(\mathcal{S})$.
Since
\[
VFF'q=GUF'q
\]
it follows from the definition of fibration that there is an object
$H(q)$ and an isomorphism
\[
\beta_{q}:UF'q\rightarrow Hq
\]
in $\mathcal{R}$ such that
\[
GHq=Vq
\]
and 
\[
G\beta_{q}=V\alpha_{q}.
\]
We can also choose $\beta$ so as to assure  that for every $p\in{\rm ob}(\mathcal{P})$
\[
HFp=Up
\]
and 
\[
\beta_{Fp}={\rm id}_{Up}.
\]
Let now $\psi\in\mathcal{Q}(q_{1},\cdots,q_{n};q)$ and define $H(\psi)$
to be the composition of the following composition scheme in $\mathcal{R}$:
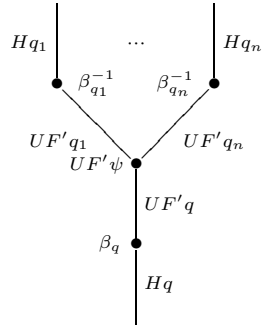
\[
\xymatrix{*{}\ar@{-}[d]_{Hq_{1}} & \ar@{}[d]|{\cdots} & *{}\ar@{-}[d]^{Hq_{n}}\\
*{\bullet}\ar@{-}[dr]_{UF'q_{1}} & \ar@{}[r]|{\beta_{q_{n}}^{-1}}\ar@{}[l]|{\beta_{q_{1}}^{-1}} & *{\bullet}\ar@{-}[dl]^{UF'q_{n}}\\
\ar@{}[r]|{UF'\psi} & *{\bullet}\ar@{-}[d]^{UF'q}\\
\ar@{}[r]|{\quad\beta_{q}} & *{\bullet}\ar@{-}[d]^{Hq}\\
 & *{}
}
\]
The resulting $H$ is easily seen to be a functor, and the desired
lift. 

For the axiom on factorizations, let $F:\mathcal{P}\rightarrow\mathcal{Q}$
be a functor. We first construct a factorization of $F$ into a trivial
cofibration followed by a fibration. Construct first the following
operad $\mathcal{P}'$ with
\[
{\rm ob}(\mathcal{P}')=\{(p,\varphi,q)\in{\rm ob}(\mathcal{P})\times\mathcal{Q}(Fp,q)\times{\rm ob}(\mathcal{Q})\mid\varphi\textrm{ is an isomorphism}\}
\]
and, for objects $(p_{1},\varphi_{1},q_{1}),\cdots,(p_{n},\varphi_{n},q_{n}),(p,\varphi,q)$,
the arrows 
\[
\mathcal{P}'((p_{1},\varphi_{1},q_{1}),\ldots,(p_{n},\varphi_{n},q_{n});(p,\varphi,q))=\mathcal{P}(p_{1},\cdots,p_{n};p)
\]
with the obvious operadic structure. If we now define $G:\mathcal{P}\rightarrow\mathcal{P}'$
on objects $p\in{\rm ob}(\mathcal{P})$ by 
\[
G(p)=(p,{\rm id}_{Fp},Fp)
\]
and for any morphism $\psi\in\mathcal{P}(p_{1},\ldots,p_{n};p)$ by
\[
G(\psi)=\psi
\]
we evidently get a functor, which is clearly a trivial cofibration.
We now define the functor $H:\mathcal{P}'\rightarrow Q$ on objects
$(p,\varphi,q)\in{\rm ob}(\mathcal{P}')$ by 
\[
H(p,\varphi,q)=q
\]
and on an arrow $\psi\in\mathcal{P}'((p_{1},\varphi_{1},q_{1}),\ldots,(p_{n},\varphi_{n},q_{n});(p,\varphi,q))$
to be the composition of the composition scheme
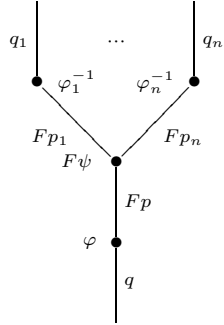
\[
\xymatrix{*{}\ar@{-}[d]_{q_{1}} & \ar@{}[d]|{\cdots} & *{}\ar@{-}[d]^{q_{n}}\\
*{\bullet}\ar@{-}[dr]_{Fp_{1}} & \ar@{}[r]|{\varphi_{n}^{-1}}\ar@{}[l]|{\varphi_{1}^{-1}} & *{\bullet}\ar@{-}[dl]^{Fp_{n}}\\
\ar@{}[r]|{F\psi} & *{\bullet}\ar@{-}[d]^{Fp}\\
\ar@{}[r]|{\quad\varphi} & *{\bullet}\ar@{-}[d]^{q}\\
 & *{}
}
\]
Clearly, $H$ is a fibration since if $f:H(p,\varphi,q)\rightarrow q'$
is an isomorphism in $\mathcal{Q}$, then $(p,f\varphi,q')$ is also
an object of $\mathcal{Q}$ and ${\rm id_{p}}$ is an isomorphism
in $\mathcal{P}'$ from $(p,\varphi,q)$ to $(p,f\varphi,q')$ which,
by definition, maps under $H$ to $f\varphi\circ F({\rm id}_{p})\circ\varphi^{-1}=f$.
Since we obviously have that $F=H\circ G$ we have the desired factorization. 

We now proceed to prove that $F$ can be factored as a composition
of a cofibration followed by a trivial fibration. Let $\mathcal{Q}'$
be the operad with 
\[
{\rm ob}(\mathcal{Q}')={\rm ob}(\mathcal{P})\coprod{\rm ob}(\mathcal{Q})
\]
and with arrows defined as follows. Given an object $x\in{\rm ob}(\mathcal{Q}')$
let (somewhat ambiguously)
\[
Fx=\left\{ \begin{array}{cc}
x, & \textrm{if }x\in{\rm ob}(\mathcal{Q}),\\
Fx, & \textrm{if }x\in{\rm ob}(\mathcal{P}).
\end{array}\right.
\]
Now, for objects $x_{1},\cdots,x_{n},x\in{\rm ob}(\mathcal{Q}')$
let 
\[
\mathcal{Q}'(x_{1},\cdots,x_{n};x)=\mathcal{Q}(Fx_{1},\cdots,Fx_{n};Fx).
\]
The operad structure is the evident one. If we now define a functor
$G:\mathcal{P}\rightarrow Q'$ for an object $p\in{\rm ob}(\mathcal{P})$
and an arrow $\psi\in\mathcal{P}(p_{1},\cdots,p_{n};p)$ by 
\[
Gp=p
\]
 and 
\[
G\psi=F\psi,
\]
then we obviously obtain a cofibration. We now define $H:\mathcal{Q}'\rightarrow\mathcal{Q}$
as follows. Given an object $x\in{\rm ob}(\mathcal{Q}')$, if $x\in{\rm ob}(\mathcal{P})$,
then we set $Hx=Fx$ and if $x\in{\rm ob}(\mathcal{Q})$, then we
set $Hx=x$ (thus in our slightly ambiguous notation we have that
$Hx=Fx$). Given an arrow $\psi\in\mathcal{Q}'(x_{1},\ldots,x_{n};x)$,
defining $H\psi=\psi$ makes $H$ into a functor, clearly full and
faithful. Moreover $H$ is a fibration as can easily be seen. Since
obviously $F=H\circ G$, the proof is complete. Note that all operads
are both fibrant and cofibrant under this model structure. 
\end{proof}
The Boardman-Vogt tensor product is of extreme importance in the theory
of operads. It only exists for symmetric operads, since the interchange
axiom can not be stated for non-symmetric operads, and without it
the resulting tensor product is, in a sense, too large, and not very
useful. The model structure on symmetric operads, as we now show,
is compatible with the Boardman-Vogt tensor product. 
\begin{thm}
The monoidal category ${\bf Ope_{\Sigma}}$ with the Boardman-Vogt
tensor product and the model structure defined above is a monoidal
model category. \end{thm}
\begin{proof}
Since all objects are cofibrant we only have to prove that given two
cofibrations $\xymatrix{F:\mathcal{P}\ar@{^{(}->}[r] & \mathcal{Q}}
$ and $\xymatrix{G:\mathcal{P}'\ar@{^{(}->}[r] & \mathcal{Q}'}
$, the push-out corner map $F\wedge G$ in the diagram
\[
\xymatrix{\mathcal{P}\otimes_{BV}\mathcal{P}'\ar[r]^{\mathcal{P}\otimes_{BV}G}\ar[d]^{F\otimes_{BV}\mathcal{P}'} & \mathcal{P}\otimes_{BV}\mathcal{Q}'\ar[dddr]^{F\otimes_{BV}\mathcal{Q}'}\ar[d]\\
\mathcal{Q}\otimes_{BV}\mathcal{P}'\ar[rrdd]_{Q\otimes_{BV}G}\ar[r] & \mathcal{K}\ar[rdd]_{F\wedge G}\\
\\
 &  & \mathcal{Q}\otimes_{BV}\mathcal{Q}'
}
\]
is a cofibration which is a trivial cofibration if $F$ is a trivial
cofibration. Since  ${\rm ob}(\mathcal{P}\otimes_{BV}\mathcal{Q})={\rm ob}(\mathcal{P})\times{\rm ob}(\mathcal{Q})$
and since ${\rm ob}:{\bf Ope_{\Sigma}}\rightarrow{\bf Set}$ commutes
with colimits, applying the functor ${\rm ob}$ to the diagram above
we obtain the diagram
\[
\xymatrix{{\rm ob}(\mathcal{P})\times{\rm ob}(\mathcal{P}')\ar[r]^{\mathcal{P}\times G}\ar[d]^{F\times\mathcal{P}'} & {\rm ob}(\mathcal{P})\times{\rm ob}(\mathcal{Q}')\ar[d]^{H}\ar[dddr]^{F\times\mathcal{Q}'}\\
{\rm ob}(\mathcal{Q})\times{\rm ob}(\mathcal{P}')\ar[r]\ar[rrdd]_{\mathcal{Q}\times G} & {\rm ob}(\mathcal{K})\ar[ddr]_{F\wedge G}\\
\\
 &  & {\rm ob}(\mathcal{Q})\times{\rm ob}(\mathcal{Q}')
}
\]
which is again a pushout. We are given that $F$ and $G$ are injective,
from which it follows that $F\times\mathcal{P}'$ and $\mathcal{P}\times G$
are also injective. It is now easy to verify that $F\wedge G$ is
injective as well which proves that the operad map $F\wedge G:\mathcal{K}\rightarrow\mathcal{Q}\otimes_{BV}\mathcal{Q}'$
is a cofibration. Assume now that $F$ in the first diagram is also
a weak equivalence, i.e., an operadic equivalence. It is trivial to
verify that $F\otimes_{BV}\mathcal{P}'$ is also an equivalence. Thus
$F\times\mathcal{P}'$ is a trivial cofibration. Since trivial cofibrations
are stable under cobase change it follows that $H$ is a trivial cofibration.
Since $F\times\mathcal{Q}'$ is also an equivalence, the three for
two property implies that $F\wedge G$ is a trivial cofibration. 
\end{proof}
Consider now ${\bf Cat}$ with the Rezk model structure and the categories
${\bf Ope}$ and ${\bf Ope_{\Sigma}}$ with the model structure given
above. Recall the adjunction between categories and operads. 
\begin{lem}
The adjunctions $\xymatrix{{\bf Ope}\ar@<-2pt>[r]_{j^{*}} & {\bf Cat}\ar@<-2pt>[l]_{j_{!}}}
$ and $\xymatrix{{\bf Ope_{\Sigma}}\ar@<-2pt>[r]_{\quad j^{*}} & {\bf Cat}\ar@<-2pt>[l]_{\quad j_{!}}}
$ are Quillen adjunctions.\end{lem}
\begin{proof}
It is enough to prove that $j_{!}$ preserves cofibrations and trivial
cofibrations. Actually it is trivial to verify the much stronger property
that both $j^{*}$ and $j_{!}$ preserve fibrations, cofibrations,
and weak equivalences. 
\end{proof}
We end our treatment of the model structure of operads with the following
explicit construction of generating cofibrations. 
\begin{thm}
The model structures on ${\bf Ope}$ and on ${\bf Ope_{\Sigma}}$
are cofibrantly generated. \end{thm}
\begin{proof}
Again, the symmetric group actions pose no difficulties, and so we
give the details in the non-symmetric case; the symmetric case obtained
by symmetrizing. Let $\star$ be the operad with one object and just
one arrow (necessarily the identity, and notice that $\star$ can
be considered as a symmetric or a non-symmetric operad) and let $H$
be the free living isomorphism operad, which has two objects and,
besides the necessary identities, just one isomorphism between the
two objects. It is a triviality to check that a functor $F:\mathcal{P}\rightarrow\mathcal{Q}$
is a fibration if, and only if,  it has the right lifting property
with respect to (any one of the two possible functors) $*\rightarrow H$. 

To characterize the trivial fibrations by right lifting properties
we will need to consider several other operads. First of all, it is
clear that if a functor $F:\mathcal{P}\rightarrow\mathcal{Q}$ has
the right lifting property with respect to $\phi\rightarrow\star$,
then $F:{\rm ob}(\mathcal{P})\rightarrow{\rm ob}(\mathcal{Q})$ is
surjective (where $\phi$ is the initial operad with no objects).
For each $n\ge1$ consider the operad $Ar_{n}$ that has $n+1$ objects
$\{0,1,\ldots,n\}$ and is generated by a single arrow from $(1,\ldots,n)$
to $0$. Thus a functor $Ar_{n}\rightarrow\mathcal{P}$ is just a
choice of an arrow in $\mathcal{P}$ of arity $n$. Let $\partial Ar_{n}$
be the sub-operad of $Ar_{n}$ that contains all the objects of $Ar_{n}$
but only the identity arrows. It now easily follows that if a functor
$F:\mathcal{P}\rightarrow\mathcal{Q}$ has the right lifting property
with respect to the inclusion $\partial Ar_{n}\rightarrow Ar_{n}$,
then for any objects $p_{1},\cdots,p_{n},p\in{\rm ob}(\mathcal{P})$,
the function 
\[
F:\mathcal{P}(p_{1},\ldots,p_{n};p)\rightarrow\mathcal{Q}(Fp_{1},\ldots,Fp_{n};Fp)
\]
is surjective. Consider now the operad $PAr_{n}$ with $n+1$ objects
$\{0,1,\ldots,n\}$ generated by two distinct arrows from $(1,\ldots,n)$
to $0$ and the obvious map $PAr_{n}\rightarrow Ar_{n}$ which identifies
these two arrows. If a functor $F:\mathcal{P}\rightarrow\mathcal{Q}$
has the right lifting property with respect to $PAr_{n}\rightarrow Ar_{n}$,
then the map 
\[
F:\mathcal{P}(p_{1},\ldots,p_{n};p)\rightarrow\mathcal{Q}(Fp_{1},\ldots,Fp_{n};Fp)
\]
 is injective. Combining these results we see that if a functor $F:\mathcal{P}\rightarrow\mathcal{Q}$
has the right lifting property with respect to the set of functors
\[
\{\phi\rightarrow\star\}\cup\{\partial Ar_{n}\rightarrow Ar_{n}\mid n\ge0\}\cup\{PAr_{n}\rightarrow Ar_{n}\mid n\ge0\},
\]
then $F$ is fully faithful and $F:{\rm ob}(\mathcal{P})\rightarrow{\rm ob}(\mathcal{Q})$
is surjective, which implies that $F$ is a trivial fibration. Finally,
since all the functors just mentioned are cofibrations it follows
that all trivial fibrations have the right lifting property with respect
to them. This then proves that the trivial fibrations are exactly
those functors having the right lifting property with respect to that
set. 
\end{proof}
Finally, we recover the Rezk model structure by exhibiting ${\bf Cat}$
as a slice category of ${\bf Ope}$ as well as of ${\bf Ope_{\Sigma}}$.
Consider again $\star$, the operad with one object and only one morphism,
necessarily the identity morphism. 
\begin{thm}
The slice category ${\bf Ope}/\star$ (resp. ${\bf Ope_{\Sigma}}/\star$)
is isomorphic to ${\bf Cat}$ and slicing the model structure on operads
yields the Rezk model structure on categories. \end{thm}
\begin{proof}
The objects of ${\bf Ope}/\star$ are functors $F:\mathcal{P}\to\star$.
Since a functor must preserve arities of morphisms, and $\star$ only
has one unary morphism, if follows that $\mathcal{P}$ only has unary
morphisms. It is now easy to see that $F\mapsto j^{*}(\mathcal{P})$
is an isomorphism of categories between ${\bf Ope}/\star$ and ${\bf Cat}$.
The claim about the model structures is immediate and the argument
for symmetric operads is similar. 
\end{proof}
\bibliographystyle{plain}
\bibliography{/Users/ittayweiss/Documents/PapersInProgress/generalReferences}

\end{document}